\documentclass[11pt]{amsart}
\usepackage[margin=1.1in]{geometry}
\usepackage{amsmath,amssymb,amsthm}
\usepackage{enumitem}
\usepackage{tikz}
\usepackage{xcolor}
\usepackage[colorlinks=true,linkcolor=blue!55!black,citecolor=blue!55!black,urlcolor=blue!55!black]{hyperref}

\newtheorem{maintheorem}{Theorem}

\newtheorem{theorem}{Theorem}[section]
\newtheorem{proposition}[theorem]{Proposition}
\newtheorem{lemma}[theorem]{Lemma}
\newtheorem{corollary}[theorem]{Corollary}
\theoremstyle{definition}
\newtheorem{definition}[theorem]{Definition}
\theoremstyle{remark}
\newtheorem{remark}[theorem]{Remark}

\numberwithin{equation}{section}

\newcommand{\RP}{\mathbb{RP}}
\newcommand{\R}{\mathbb{R}}
\newcommand{\Z}{\mathbb{Z}}
\newcommand{\PP}{\mathbb{P}}
\newcommand{\tr}{\operatorname{tr}}
\newcommand{\cA}{\mathcal{A}}
\newcommand{\GL}{\mathrm{GL}(3,\R)}
\newcommand{\SL}{\mathrm{SL}(3,\R)}
\newcommand{\ip}[2]{\langle #1,#2\rangle}

\title[Point-balanced arrangements and the Hirzebruch property]{Point-balanced arrangements in the real projective plane and the Hirzebruch property}
\author{M. de Borbon}
\address{Department of Mathematical Sciences, Loughborough University, Loughborough LE11 3TU, United Kingdom}
\email{m.de-borbon@lboro.ac.uk}
\date{}

\begin{document}

\begin{abstract}
A collection of $m \geq 2$ lines through the origin in $\mathbb{R}^2$ is balanced if the sum of the orthogonal projections onto these lines equals $m/2$ times the identity. A line arrangement in $\mathbb{RP}^2$, endowed with the round metric of curvature $1$, is point-balanced if the tangent lines at every vertex form a balanced collection. A line arrangement $\mathcal{A}$ in $\mathbb{RP}^2$ has the Hirzebruch property if it consists of $3k$ lines and every line contains exactly $k+1$ vertices.

Using a Kempf--Ness convexity argument, we show that an irreducible arrangement $\mathcal{A}$ has the Hirzebruch property if and only if its projective equivalence class contains a point-balanced representative, unique up to orthogonal transformations. We then combine elementary properties of balanced collections of lines in $\mathbb{R}^2$ with spherical geometry to prove that every irreducible point-balanced arrangement is a reflection arrangement. This gives a new proof of Panov's classification of real Hirzebruch arrangements.
\end{abstract}

\maketitle

\section{Introduction}\label{sec:intro}

\subsection{Main results}
An \emph{arrangement} is a finite set \(\cA\) of lines in \(\RP^2\) that is \emph{essential}: its lines do not all pass through a common point, or equivalently, the linear forms defining them span \((\R^3)^*\). A \emph{vertex} of \(\cA\) is a point lying on at least two of its lines. The arrangement is a \emph{near-pencil} if all of its lines but one pass through a common point. We say that \(\cA\) is \emph{irreducible} if it is not a near-pencil. 
A matrix \(F\in\GL\) acts on \(\RP^2\) by \(F[x]=[Fx]\) and maps lines to lines, so \(F(\cA)=\{F(C):C\in\cA\}\) is again an arrangement. Two arrangements \(\cA_0\) and \(\cA\) are \emph{projectively equivalent} if \(\cA=F(\cA_0)\) for some \(F\in\GL\), and \emph{\(O(3)\)-equivalent} if \(F\) can be chosen in \(O(3)\).

Following \cite{Hirzebruch85,Panov}, we consider the following class of arrangements.

\begin{definition}\label{def:hirzebruch}
The arrangement \(\cA\) has the \emph{Hirzebruch property}, or is a \emph{Hirzebruch arrangement}, if the number of lines is \(3k\) and every line of \(\cA\) contains exactly \(k+1\) vertices of \(\cA\).
\end{definition}

We endow \(\RP^2\) with the round metric of curvature \(1\) induced by the standard Euclidean inner product on \(\R^3\).

\begin{definition}\label{def:pointbalanced}
Let \(v\) be a vertex of the arrangement \(\cA\), let \(\cA_v\) be the set of lines of \(\cA\) through \(v\), and set \(V=T_v\RP^2\), with identity \(I_V\). For \(C\in\cA_v\), let \(\mathrm{pr}_{T_vC}\colon V\to V\) be the orthogonal projection onto the tangent line \(T_vC\) of \(C\) at \(v\). We say that \(\cA\) is \emph{balanced at \(v\)} if
\[
\sum_{C\in\cA_v}\mathrm{pr}_{T_vC}=\frac{|\cA_v|}2\,I_V .
\]
The arrangement \(\cA\) is \emph{point-balanced} if it is balanced at every vertex.
\end{definition}

\begin{maintheorem}\label{thm:correspondence}
Let \(\cA\) be an irreducible arrangement. Then the following are equivalent:
\begin{enumerate}[label=\textup{(\roman*)}]
\item \(\cA\) has the Hirzebruch property;
\item \(\cA\) is projectively equivalent to a point-balanced arrangement.
\end{enumerate}
\end{maintheorem}

Let \(W\subset O(3)\) be a finite group generated by reflections. Its \emph{reflection arrangement} is the set of projective lines \(\PP(\Pi)\subset\RP^2\), where \(\Pi\subset\R^3\) runs over the reflecting planes of \(W\).     

\begin{maintheorem}\label{thm:balanced}
Every irreducible point-balanced arrangement is a reflection arrangement.
\end{maintheorem}

\begin{corollary}\label{cor:panov}
Suppose that \(\cA\) is a Hirzebruch arrangement. Then \(\cA\) is projectively equivalent to the reflection arrangement of \((A_1)^3\), \(A_3\), \(B_3\) or \(H_3\).
\end{corollary}

The reflection arrangements of \(A_3\), \(B_3\) and \(H_3\) are shown in Figure~\ref{fig:arrangements}. They consist of the lines \(\PP(\Pi)\), where \(\Pi\) runs over the planes of symmetry of the regular tetrahedron, the cube (or octahedron) and the icosahedron (or dodecahedron), respectively, centred at the origin. Joining the centre of each face of the solid to its vertices and to the midpoints of its edges, and projecting radially onto \(S^2\), gives a triangulation of \(S^2\) by congruent spherical triangles with angles \(\pi/2,\pi/3,\pi/m\), where \(m=3,4,5\), respectively. These triangles are the fundamental domains of the reflection group, and their images in \(\RP^2\) are the faces of the arrangement.

\begin{figure}[ht]
\centering
\input{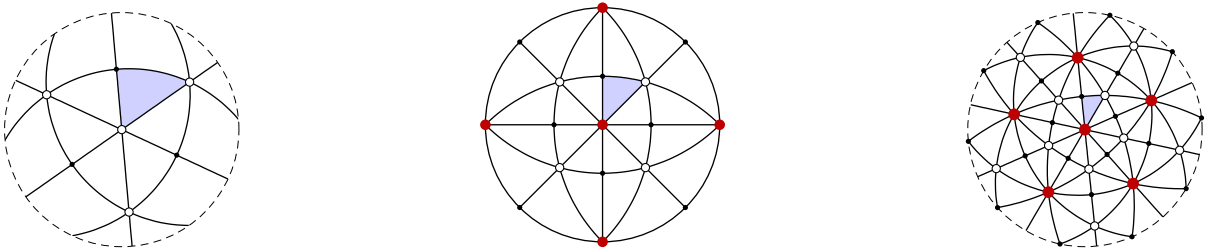}

\caption{The reflection arrangements of \(A_3\), \(B_3\) and \(H_3\) under stereographic projection of a closed hemisphere of \(S^2\) onto the closed unit disc; identifying antipodal boundary points gives \(\RP^2\). Vertices on two lines are small dots, vertices on three lines are open circles, and vertices on four or five lines are red. The shaded face has angles \(\pi/2,\pi/3,\pi/m\) for \(m=3,4,5\). For \(B_3\) the boundary circle is itself a line of the arrangement.}
\label{fig:arrangements}
\end{figure}

\subsection{Strategy and relation to other work}\label{subsec:strategy}

\emph{Theorem~\ref{thm:correspondence}.} The main tool is the Kempf--Ness convexity argument. Proposition~\ref{prop:kempf-ness} gives \(F\in\SL\), unique up to left multiplication by \(SO(3)\), such that \(F(\cA)\) is balanced (Definition~\ref{def:balanced}). By the local-to-global balance identity \eqref{eq:balance-identity}, an irreducible arrangement of \(n\) lines is point-balanced if and only if it is balanced and every line contains \((n+3)/3\) vertices (Proposition~\ref{prop:balance-criterion}). This proves Theorem~\ref{thm:correspondence}. 

\emph{Theorem~\ref{thm:balanced}.} We use Girard's formula (Lemma~\ref{lem:girard}) and the angle estimates for balanced pencils in \(\R^2\) (Lemma~\ref{lem:balanced-pencils}) to show that every face of an irreducible point-balanced arrangement \(\cA\) is a triangle, exactly one of whose vertices is a double point of \(\cA\) (Propositions~\ref{prop:triangles} and~\ref{prop:oneright}). The main step is Proposition~\ref{prop:equal-spacing}: the lines through every vertex are equally spaced. Hence the reflection in each line of \(\cA\) preserves \(\cA\), and these reflections generate a finite reflection group whose reflection arrangement is \(\cA\) (Proposition~\ref{prop:coxeter}).

\emph{Corollary~\ref{cor:panov} and comparison with Panov.} Corollary~\ref{cor:panov} follows from Theorems~\ref{thm:correspondence} and~\ref{thm:balanced} and the classification of irreducible finite reflection groups of rank three \cite[\S2.4 and Thm.~2.7]{Humphreys}. Panov's Theorem~1.1 \cite{Panov} states that there are exactly four real Hirzebruch arrangements. This statement omits an equivalence relation: any three non-concurrent lines form a Hirzebruch arrangement, and there are infinitely many such triples. In our case, the projective equivalence in Corollary~\ref{cor:panov} comes from Kempf--Ness (Proposition~\ref{prop:kempf-ness}). Panov's proof relies on the existence and uniqueness of a polyhedral K\"ahler metric on \(\mathbb{CP}^2\) with cone singularities along the complexified arrangement, whereas the methods here are elementary and self-contained.

\subsection*{Acknowledgments}
I thank Dima Panov and Sasha Veselov. I developed the proof of
Theorem~\ref{thm:correspondence} during our collaboration on
\cite{dBPV}.

\subsection*{Use of AI}
I used Claude and ChatGPT to produce and understand some
of the arguments in this paper. The main ideas, including the
concept of point-balanced arrangements and the correspondence
with Hirzebruch arrangements, are my own.

\section{Background}\label{sec:background}

We recall the Veronese embeddings and the equivalent formulations of balance in \S\S\ref{subsec:veronese}--\ref{subsec:balanced-arr-P2}. We then give the Kempf--Ness convexity argument used in the proof of Theorem~\ref{thm:correspondence}, and the cell decomposition and Girard's spherical area formula used in the proof of Theorem~\ref{thm:balanced}.

\subsection{The Veronese embedding and tight frames}\label{subsec:veronese}

Let \(\operatorname{Sym}_0(\R^{d+1})\) denote the space of trace-free symmetric \((d+1)\times(d+1)\) matrices, with the inner product \(\ip{A}{B}=\tr(AB)\). For \(d\ge1\), the Veronese embedding of \(\RP^d\) in traceless form is
\begin{equation}\label{eq:veronese-general}
\nu_d:\RP^d\longrightarrow\operatorname{Sym}_0(\R^{d+1}),
\qquad
\nu_d([u])=uu^T-\frac1{d+1}\,I_{d+1},
\end{equation}
where \(u\) is a unit vector and \(I_k\) denotes the \(k\times k\) identity matrix. The image of \(\nu_d\) lies on the sphere of radius \(\sqrt{d/(d+1)}\) centred at the origin of \(\operatorname{Sym}_0(\R^{d+1})\).

Unit vectors \(u_1,\dots,u_m\in\R^{d+1}\) form a \emph{unit-norm tight frame} for \(\R^{d+1}\) if \(\sum_ju_ju_j^T=\frac m{d+1}\,I_{d+1}\); see \cite{Waldron}. Since
\begin{equation}\label{eq:veronese-frame}
\sum_{j=1}^m\nu_d([u_j])=\sum_{j=1}^mu_ju_j^T-\frac m{d+1}\,I_{d+1},
\end{equation}
pairwise non-proportional unit vectors \(u_1,\dots,u_m\in\R^{d+1}\) form a unit-norm tight frame for \(\R^{d+1}\) if and only if the barycentre of the points \(\nu_d([u_1]),\dots,\nu_d([u_m])\) is the origin.

\emph{The case \(d=1\)}.
We identify \(\R^2\) with \(\mathbb C\) by \((s,t)\mapsto s+it\), so that the unit vectors are the \(e^{i\theta}=(\cos\theta,\sin\theta)\) and \(\RP^1=S^1/\{\pm1\}\) consists of the points \([e^{i\theta}]\), \(\theta\in\R/\pi\Z\). For \(d=1\), \eqref{eq:veronese-general} is
\begin{equation}\label{eq:veronese-RP1}
\nu_1:\RP^1\longrightarrow\operatorname{Sym}_0(\R^2),
\qquad
\nu_1([u])=uu^T-\frac12I_2 .
\end{equation}
Put
\[
E_1=\begin{pmatrix}1&0\\0&-1\end{pmatrix},
\qquad
E_2=\begin{pmatrix}0&1\\1&0\end{pmatrix}.
\]
These are orthogonal, with \(\|E_1\|^2=\|E_2\|^2=2\), so \(E_1/\sqrt2,E_2/\sqrt2\) is an orthonormal basis of \(\operatorname{Sym}_0(\R^2)\). Expanding \(uu^T\) for \(u=e^{i\theta}\) gives
\begin{equation}\label{eq:veronese-RP1-angle}
\nu_1([e^{i\theta}])=\frac12\big(\cos2\theta\,E_1+\sin2\theta\,E_2\big).
\end{equation}
We identify \(\operatorname{Sym}_0(\R^2)\) with \(\R^2\) by means of the orthonormal basis \(E_1/\sqrt2,E_2/\sqrt2\), that is, \((sE_1+tE_2)/\sqrt2\mapsto(s,t)\), and then with \(\mathbb C\) by \((s,t)\mapsto s+it\). Under these identifications, \eqref{eq:veronese-RP1-angle} reads \(\nu_1([e^{i\theta}])=\frac1{\sqrt2}\,e^{2i\theta}\). Hence \(\nu_1\) maps \(\RP^1\) onto the circle of radius \(1/\sqrt2\) centred at the origin, and \(\sqrt2\,\nu_1\) is the identification
\begin{equation}\label{eq:directions-circle}
\RP^1\longrightarrow S^1,
\qquad
[e^{i\theta}]\longmapsto e^{2i\theta},
\end{equation}
induced by the squaring map on \(S^1\), which is constant on antipodal pairs. Hence, for \(\theta_1,\dots,\theta_m\in\R/\pi\Z\), \(\sum_je^{2i\theta_j}=0\) if and only if the barycentre of the points \(\nu_1([e^{i\theta_j}])\) is the origin.

\subsection{Pencils of lines in the Euclidean plane}\label{subsec:pencils}
Equip \(\R^2\) with the standard inner product.
In this subsection a \emph{line} in \(\R^2\) is a one-dimensional linear subspace. For a line \(\ell\subset\R^2\) let \(\mathrm{pr}_\ell\colon\R^2\to\R^2\) be the orthogonal projection onto \(\ell\).

\begin{definition}\label{def:balanced-pencil}
	Distinct lines \(\ell_1,\dots,\ell_m\subset\R^2\), \(m\ge2\), form a \emph{balanced pencil} if
	\[
	\sum_{j=1}^m\mathrm{pr}_{\ell_j}=\frac m2\,I_2 .
	\]
\end{definition}

\begin{definition}\label{def:direction}
Every line \(\ell\subset\R^2\) is uniquely of the form \(\ell=\R e^{i\theta}\) with \(\theta\in\R/\pi\Z\). We call \(\theta\) the \emph{direction} of \(\ell\); thus \(\ell\) is represented by the point \([e^{i\theta}]\) of \(\RP^1\).
\end{definition}

\begin{lemma}\label{lem:balance-characterisations}
Let \(\ell_1,\dots,\ell_m\), \(m\ge2\), be distinct lines in \(\R^2\), with directions \(\theta_j\), and let \(n_j\in\R^2\) be a unit vector orthogonal to \(\ell_j\). The following are equivalent:
\begin{enumerate}[label=\textup{(\roman*)}]
\item \(\ell_1,\dots,\ell_m\) form a balanced pencil;
\item the directions satisfy
\begin{equation}\label{eq:point-balance}
\sum_{j=1}^me^{2i\theta_j}=0 ;
\end{equation}
\item the barycentre of the points \(\nu_1(\ell_j)\) given by \eqref{eq:veronese-RP1} is the origin of \(\operatorname{Sym}_0(\R^2)\);
\item the vectors \(n_1,\dots,n_m\) form a unit-norm tight frame for \(\R^2\).
\end{enumerate}
\end{lemma}

\begin{proof}
Put \(t_j=e^{i\theta_j}\), so that \(\ell_j=\R t_j\) and \(\mathrm{pr}_{\ell_j}=t_jt_j^T=\frac12I_2+\nu_1(\ell_j)\) by \eqref{eq:veronese-RP1}. Put \(Z=\sum_je^{2i\theta_j}\). By \eqref{eq:veronese-RP1-angle},
\[
\sum_{j=1}^m\mathrm{pr}_{\ell_j}=\frac m2\,I_2+\sum_{j=1}^m\nu_1(\ell_j)=\frac m2\,I_2+\frac12\big(\operatorname{Re}Z\,E_1+\operatorname{Im}Z\,E_2\big).
\]
Since \(E_1,E_2\) are linearly independent, (i) holds if and only if \(Z=0\), and if and only if \(\sum_j\nu_1(\ell_j)=0\). This proves that (i), (ii) and (iii) are equivalent. Since \(t_j,n_j\) is an orthonormal basis of \(\R^2\), we have \(n_jn_j^T=I_2-t_jt_j^T\). Therefore \(\sum n_jn_j^T = (m/2)I_2\) if and only if \(\sum \mathrm{pr}_{\ell_j} = (m/2) I_2\); showing that (iv) is equivalent to (i).
\end{proof}

\begin{remark}
Equation \eqref{eq:point-balance} is the criterion for a tight frame for \(\R^2\) in \cite[Exercise~2.9(a)]{Waldron}.	
\end{remark}

\begin{remark}\label{rmk:O2-invariance}
	An element of \(O(2)\) maps lines with directions \(\theta_j\) to lines with directions \(\varepsilon\theta_j+c\), where \(\varepsilon\in\{\pm1\}\), and therefore transforms \(\sum_je^{2i\theta_j}\) into either \(e^{2ic}\sum_je^{2i\theta_j}\) or \(e^{2ic}\,\overline{\sum_je^{2i\theta_j}}\). In particular, the modulus of \(Z:=\sum_je^{2i\theta_j}\) is invariant under \(O(2)\).
\end{remark}

\begin{definition}\label{def:equally-spaced}
Let \(\ell_1,\dots,\ell_m\subset\R^2\), \(m\ge2\), be distinct lines with directions represented by \(\theta_1<\theta_2<\dots<\theta_m<\theta_1+\pi\). The \emph{angles between consecutive lines} are
\[
g_j=\theta_{j+1}-\theta_j\quad(1\le j<m),\qquad g_m=\theta_1+\pi-\theta_m ;
\]
the lines divide \(\R^2\) into \(2m\) sectors, whose angles are the \(g_j\), each listed twice, and \(g_1+\dots+g_m=\pi\). The lines are \emph{equally spaced} if all \(g_j\) equal \(\pi/m\); then the \(2\theta_j\) are the vertices of a regular \(m\)-gon, so equally spaced lines form a balanced pencil by Lemma~\ref{lem:balance-characterisations}.	
\end{definition}

\subsection{Balanced arrangements in the projective plane}\label{subsec:balanced-arr-P2}

We use the round metric on \(\RP^2\) obtained as the quotient of the unit sphere \(S^2\subset\R^3\) by \(\pm1\): the distance between \([x]\) and \([y]\) is the angle in \([0,\pi/2]\) between the lines \(\R x\) and \(\R y\), projective lines are closed geodesics of length \(\pi\), and every isometry is induced by an element of \(O(3)\).

We now transfer the notions of \S\ref{subsec:pencils} from lines in \(\R^2\) to projective lines through a point of \(\RP^2\). Let \(v=[x]\), with \(x\) a unit vector. The differential at \(x\) of the quotient map \(S^2\to\RP^2\) identifies \(T_v\RP^2\) isometrically with \(x^\perp\). A projective line \(C\) through \(v\) is of the form \(C=\PP(\operatorname{span}(x,t))\) for a unit vector \(t\in x^\perp\), unique up to sign, and its tangent line at \(v\) is \(T_vC=\R t\). Choose an orthonormal basis \(a,b\) of \(x^\perp\) and let \(J_v\) be the \(3\times2\) matrix with columns \(a,b\). Then \(J_v\) maps \(\R^2\) isometrically onto \(x^\perp=T_v\RP^2\), with \(J_v^TJ_v=I_2\) and \(J_vJ_v^T=\mathrm{pr}_{x^\perp}=I_3-xx^T\), the orthogonal projection onto \(x^\perp\).

\begin{definition}\label{def:pencil-at-point}
For a projective line \(C\) through \(v\), let \(\ell_C=J_v^T(T_vC)\), a line in \(\R^2\). The \emph{direction} of \(C\) at \(v\), with respect to \(a,b\), is the direction of \(\ell_C\), that is, the \(\theta\in\R/\pi\Z\) with \(C=\PP\big(\operatorname{span}(x,\cos\theta\,a+\sin\theta\,b)\big)\). Distinct projective lines \(C_1,\dots,C_m\) through \(v\) form a \emph{balanced pencil at \(v\)} if \(\ell_{C_1},\dots,\ell_{C_m}\) form a balanced pencil in \(\R^2\) as in Definition~\ref{def:balanced-pencil}, and the \emph{angles between consecutive lines at \(v\)} are those of \(\ell_{C_1},\dots,\ell_{C_m}\).
\end{definition}

\begin{remark}
Another choice of \(a,b\) replaces \(J_v\) by \(J_vg\) with \(g\in O(2)\), which changes neither balanced pencils at \(v\) nor the angles between consecutive lines at \(v\).	
\end{remark}

Let \(\cA\) be an arrangement of \(n\) lines in \(\RP^2\). 
Recall that \(\cA_v\) is the set of lines of \(\cA\) through a vertex \(v\); its cardinality \(m(v)=|\cA_v|\) is the \emph{multiplicity} of \(v\). For \(C\in\cA\) choose a unit vector \(n_C\in\R^3\), unique up to sign, with \(C=\PP(n_C^\perp)\). The points \([n_C] \in \RP^2\) are the poles of the lines of \(\cA\); they form the \emph{dual point configuration} of \(\cA\).

\begin{lemma}\label{lem:vertex-balance}
Let \(v=[x]\) be a vertex of \(\cA\). The following are equivalent:
\begin{enumerate}[label=\textup{(\roman*)}]
\item \(\cA\) is balanced at \(v\) as in  Definition~\ref{def:pointbalanced};
\item the lines of \(\cA_v\) form a balanced pencil at \(v\) as in Definition~\ref{def:pencil-at-point};
\item \(\displaystyle\sum_{C\in\cA_v}n_Cn_C^T=\frac{m(v)}2\,\mathrm{pr}_{x^\perp}\).
\end{enumerate}
\end{lemma}

\begin{proof}
For \(C\in\cA_v\), the map \(J_v^T\) restricts to an isometry \(x^\perp\to\R^2\) carrying \(T_vC\) onto \(\ell_C\), so (i) and (ii) are equivalent. Since \(v\in C=\PP(n_C^\perp)\), the vector \(n_C\) lies in \(x^\perp\) and is a unit normal to the line \(T_vC\) in \(x^\perp\); hence \(\mathrm{pr}_{T_vC}=I_{x^\perp}-n_Cn_C^T\) on \(x^\perp\). Summing over \(C\in\cA_v\), and using that each \(n_Cn_C^T\) vanishes on \(\R x\), shows that (i) and (iii) are equivalent.
\end{proof}

For a vertex \(v=[x]\), with the directions \(\theta_C\) at \(v\) of the lines \(C\in\cA_v\) as in Definition~\ref{def:pencil-at-point}, put
\begin{equation}\label{eq:Zv}
Z_v=\sum_{C\in\cA_v}e^{2i\theta_C}.
\end{equation}
By Remark~\ref{rmk:O2-invariance}, \(|Z_v|\) does not depend on the choice of \(a,b\), and by Lemmas~\ref{lem:vertex-balance} and~\ref{lem:balance-characterisations}, \(\cA\) is point-balanced if and only if \(Z_v=0\) for every vertex \(v\).

\begin{definition}\label{def:balanced}
An arrangement \(\cA\) of \(n\) lines in \(\RP^2\) is \emph{balanced} if its unit normals form a unit-norm tight frame for \(\R^3\), that is,
\begin{equation}\label{eq:global-balance}
M=\frac n3\,I_3, \quad 	M=\sum_{C\in\cA}n_Cn_C^T .
\end{equation}
Equivalently, by \eqref{eq:veronese-frame}, if the barycentre of the points \(\nu_2([n_C]) \in \operatorname{Sym}_0(\R^3)\) is the origin.
\end{definition}

\subsection{The Kempf--Ness convexity argument}\label{subsec:kempf-ness}

We recall a particular instance of the standard Kempf--Ness criterion characterizing the existence of a zero of the moment map in terms of stability via the minimization of a convex function; see \cite[\S5.4, Theorem~5.19]{Szekelyhidi}.

\begin{proposition}\label{prop:kempf-ness}
Let \(\cA\) be an arrangement of \(n\ge4\) lines in \(\RP^2\) such that every vertex has multiplicity \(< \frac{2n}3\). Then there is \(F\in\SL\), unique up to left multiplication by \(SO(3)\), such that \(F(\cA)\) is balanced.
\end{proposition}

For \(F\in\SL\) put \(B=F^{-T}\). The unit normals of \(F(\cA)\) are the vectors \(y_C=Bn_C/|Bn_C|\), so \(F(\cA)\) is balanced if and only if
\begin{equation}\label{eq:moment}
\sum_{C\in\cA}y_Cy_C^T=\frac n3\,I_3 .
\end{equation}
Since \(F\mapsto F^{-T}\) is a bijection of \(\SL\) and \((UF)^{-T}=UF^{-T}\) for \(U\in SO(3)\), Proposition~\ref{prop:kempf-ness} states that there is \(B\in\SL\), unique up to left multiplication by \(SO(3)\), satisfying \eqref{eq:moment}. We prove this by the Kempf--Ness convexity argument on \(SO(3)\backslash\SL\). Put
\begin{equation}\label{eq:KN-functional}
\Phi(B)=\sum_{C\in\cA}\log|Bn_C|,\qquad B\in\SL ,
\end{equation}
which is invariant under \(B\mapsto UB\), \(U\in SO(3)\). Let \(X\ne0\) be a traceless symmetric matrix, with eigenvalues \(\lambda_1\ge\lambda_2\ge\lambda_3\) and orthonormal eigenvectors \(e_1,e_2,e_3\), and put \(\varphi(t)=\Phi(e^{tX}B)\). Then
\begin{equation}\label{eq:KN-geodesic}
2\varphi(t)=\sum_{C\in\cA}\log\Big(\sum_i\ip{Bn_C}{e_i}^2e^{2t\lambda_i}\Big),\qquad \varphi'(0)=\tr\Big(X\sum_{C\in\cA}y_Cy_C^T\Big).
\end{equation}

\begin{proof}[Proof of Proposition~\textup{\ref{prop:kempf-ness}}]
\emph{Convexity.} Each summand in \eqref{eq:KN-geodesic} is of log-sum-exp type. Its second derivative is the variance of the numbers \(2\lambda_i\) for the weights proportional to \(\ip{Bn_C}{e_i}^2e^{2t\lambda_i}\), so it is either positive for all \(t\) or identically zero. The second derivative is identically zero exactly when \(Bn_C\) lies in a single eigenspace of \(X\). Suppose that every \(Bn_C\) lies in an eigenspace of \(X\). Since \(X\ne0\) is traceless, it is not scalar. If \(X\) has three \(1\)-dimensional eigenspaces, there are at most three points \([Bn_C]\), contradicting \(n\ge4\). Otherwise \(X\) has a \(2\)-dimensional eigenspace \(E\), and all but at most one of the normals \(n_C\) lie in the plane \(B^{-1}E=x^\perp\), for some unit vector \(x\); that is, all but at most one of the lines of \(\cA\) pass through \([x]\), contradicting the hypothesis since \(n-1\ge\frac{2n}3\). Hence \(\varphi''>0\) everywhere: \(\Phi\) is strictly convex along \(t\mapsto e^{tX}B\).

\emph{Critical points.} By \eqref{eq:KN-geodesic}, \(\varphi'(0)=0\) for every traceless symmetric \(X\) if and only if \(\sum_Cy_Cy_C^T\) is scalar. Since \(\tr\sum_Cy_Cy_C^T=n\), a scalar matrix of this form equals \(\frac n3I_3\); this is \eqref{eq:moment}.

\emph{Uniqueness.} Let \(B_0,B_1\) satisfy \eqref{eq:moment}. By polar decomposition \(B_1B_0^{-1}=Ue^X\) with \(U\in SO(3)\) and \(X\) traceless symmetric. Since \eqref{eq:moment} is preserved by \(B\mapsto UB\), the matrix \(e^XB_0=U^{-1}B_1\) also satisfies \eqref{eq:moment}. If \(X\ne0\), the strictly convex function \(t\mapsto\Phi(e^{tX}B_0)\) would have vanishing derivative at \(t=0\) and \(t=1\), which is impossible. Hence \(X=0\) and \(B_1=UB_0\).

\emph{Existence.} By polar decomposition and \(SO(3)\)-invariance it suffices to minimise \(f(X)=\Phi(e^X)\) over traceless symmetric \(X\). Let \(X\ne0\) be traceless symmetric, with eigenvalues \(\lambda_1\ge\lambda_2\ge\lambda_3\) and orthonormal eigenvectors \(e_1,e_2,e_3\). Let \(n_2\) be the number of normals \(n_C\) in the plane \(e_1^\perp\), that is, the number of lines of \(\cA\) through \([e_1]\), and \(n_1\) the number on the line \(\R e_3\). As \(t\to+\infty\), each summand of \eqref{eq:KN-geodesic} with \(B=I_3\) is \(2t\lambda_i+O(1)\), where \(i\) is the least index with \(\ip{n_C}{e_i}\ne0\). Hence \(f(tX)=t\,s(X)+O(1)\), where, using \(\lambda_1+\lambda_2+\lambda_3=0\),
\[
s(X)=\lambda_1(n-n_2)+\lambda_2(n_2-n_1)+\lambda_3n_1
=(\lambda_1-\lambda_2)\Big(\frac{2n}3-n_2\Big)+(\lambda_2-\lambda_3)\Big(\frac n3-n_1\Big).
\]
Here \(n_2<\frac{2n}3\) by hypothesis and \(n_1\le1<\frac n3\), and \(\lambda_1>\lambda_3\), so \(s(X)>0\). 

The sublevel set \(K=\{X:f(X)\le f(0)\}\)
is closed. If \(X\in K\) and \(0\le s\le 1\), then
\[
f(sX)\le (1-s)f(0)+sf(X)\le f(0),
\]
since \(t\mapsto f(tX)\) is convex. Thus \(K\) is star-shaped about \(0\).
If \(K\) were unbounded, choose \(X_j\in K\) with \(\|X_j\|\to\infty\) and, after passing to a subsequence, \(X_j/\|X_j\|\to X\). For fixed \(t>0\) we have \(tX_j/\|X_j\|\in K\) for large \(j\), so \(tX\in K\) for every \(t>0\), contradicting \(s(X)>0\). Hence \(K\) is compact, \(f\) attains its minimum on \(K\), and this is a global minimum of \(\Phi\). A minimiser is a critical point, so it satisfies \eqref{eq:moment}.
\end{proof}

\begin{proposition}\label{prop:uniqueness}
Let \(\cA\) be an irreducible arrangement, and let \(F_0,F_1\in\GL\) be such that \(F_0(\cA)\) and \(F_1(\cA)\) are balanced \textup{(}Definition~\textup{\ref{def:balanced})}. Then \(F_1\in\R^\times O(3)\,F_0\). Conversely, \(F(\cA)\) is balanced for every \(F\in\R^\times O(3)\,F_0\).
\end{proposition}

\begin{proof}
The converse is clear, since for \(U\in O(3)\) the matrix \(M\) of \(U(\cA)\) is \(UMU^T\). Replacing \(F_i\) by \((\det F_i)^{-1/3}F_i\), we may assume \(F_i\in\SL\). The uniqueness argument in the proof of Proposition~\ref{prop:kempf-ness} uses only the strict convexity of \(t\mapsto\Phi(e^{tX}B)\) for traceless symmetric \(X\ne0\), and the convexity argument there uses only that \(n\ge4\) and that no point of \(\RP^2\) lies on all but at most one of the lines of \(\cA\). Both hold for an essential irreducible arrangement. Hence \(F_1=UF_0\) with \(U\in SO(3)\).
\end{proof}

\begin{remark}\label{rem:GIT}
If \(\cA=\{C_1,\dots,C_n\}\) and
\(\widehat{n}=n_{C_1}\otimes\cdots\otimes n_{C_n}\), then
\[
\log |B^{\otimes n}\widehat{n}|^2
=2\sum_{C\in\cA}\log|Bn_C|
=2\Phi(B),
\]
so, up to the factor \(2\), \(\Phi\) is the usual Kempf--Ness function
for the Segre linearisation, with \(\SL\) acting diagonally on
\((\R^3)^{\otimes n}\). Its critical-point equation is
precisely \eqref{eq:moment}. If \(X\) has integer eigenvalues, then
\(s(X)\) is the Hilbert--Mumford weight of \(\widehat n\) for the
one-parameter subgroup \(\tau\mapsto\tau^X\). The hypothesis of
Proposition~\ref{prop:kempf-ness} is precisely the equal-weight
Hilbert--Mumford stability condition for the dual point configuration
of \(\cA\); see \cite[Theorem~11.2]{Dolgachev}.

The same problem of finding a linear transformation for which the
normalised vectors form a unit-norm tight frame is treated in
\cite[Lemma~2.4, Proposition~2.7 and Appendix~A.5]{ArtsteinAvidanKaplanSharir} by a different
convex minimisation. For equal weights their function is
\[
t\longmapsto
\log\det\Big(\sum_{C\in\cA} e^{t_C}n_Cn_C^T\Big)
-\frac{3}{n}\sum_{C\in\cA} t_C .
\]
If \(t\) is a minimiser and
\(Q(t)=\sum_C e^{t_C}n_Cn_C^T\), then \(Q(t)^{-1/2}\), followed by
normalisation of the resulting vectors, gives a unit-norm tight frame.
Thus the two convex functions give different formulations of the same
tight-frame problem.

There is also a geometric interpretation on the symmetric space
\(\SL/SO(3)\). For a unit vector \(u\in\R^3\), the function
\(b_{[u]}(g\,SO(3))=\log|g^{-1}u|\) is, for a suitably normalised
invariant metric, a Busemann function for the point of the ideal
boundary whose stabiliser in \(\SL\) is the stabiliser of the line
\(\R u\); see \cite[Example~2.2]{KapovichLeebMillson}. Thus
\(\Phi(B)=\sum_{C\in\cA}b_{[n_C]}(B^{-1}SO(3))\), and stability implies
that this sum of Busemann functions has a unique minimum; see
\cite[\S\S3.2--3.4]{KapovichLeebMillson}. If \(B\) minimises \(\Phi\),
the minimum is at \(B^{-1}SO(3)\), and left multiplication by \(B\)
moves it to \(SO(3)\), the point corresponding to the standard
Euclidean inner product on \(\R^3\).
\end{remark}

\subsection{Cell decomposition and spherical polygons}\label{subsec:cells}

In this subsection \(\cA\) is an arbitrary arrangement of lines in \(\RP^2\). For background on real projective line arrangements see \cite{GrunbaumAS}.

\begin{definition}\label{def:cells}
The \emph{edges} of \(\cA\) are the connected components of \(C\setminus\{\text{vertices}\}\), for \(C\in\cA\), and the \emph{faces} are the connected components of \(\RP^2\setminus\bigcup_{C\in\cA}C\).

Let \(F\) be a face and let \(v\in\overline F\) be a vertex. In a small disc around \(v\) meeting no line of \(\cA\) other than those through \(v\), the lines through \(v\) bound \(2m(v)\) \emph{sectors} at \(v\), each between two consecutive lines through \(v\); since \(v\in\overline F\), the face \(F\) contains at least one of them. The \emph{angle of \(F\) at \(v\)}, denoted \(\angle_vF\), is the sum of the angles of the sectors at \(v\) contained in \(F\).
\end{definition}

\begin{definition}\label{def:rays-neighbours}
A \emph{ray} at a vertex \(x\) is one of the \(2m(x)\) local directions in which a line of \(\cA\) leaves \(x\); each line through \(x\) gives two opposite rays. The first vertex encountered along a ray \(\rho\) is denoted \(y(\rho)\) and is called the \emph{neighbour} of \(x\) along \(\rho\). The closure of the edge leaving \(x\) along \(\rho\) is a segment from \(x\) to \(y(\rho)\), denoted \([x,y(\rho)]\).
\end{definition}

\begin{lemma}\label{lem:cells}
Let \(\cA\) be an arrangement of lines in \(\RP^2\). Then:
\begin{enumerate}[label=\textup{(\alph*)}]
\item The closure of every face is the image, under the quotient map \(S^2\to\RP^2\), of a compact convex spherical polygon contained in an open hemisphere. The restriction of the quotient map to this polygon is injective and locally isometric, and its sides map to closures of edges. If \(F\) is a face and \(v\in\overline F\) is a vertex, the interior angle of this polygon at \(v\) is \(\angle_vF\).
\item Near a vertex \(v\), the lines through \(v\) cut out \(2m(v)\) sectors whose angles are the angles between consecutive lines at \(v\), each occurring twice. These sectors lie in \(2m(v)\) distinct faces. In particular, if \(F\) is a face with \(v\in\overline F\), then \(F\) contains exactly one sector at \(v\), the \emph{sector of \(F\) at \(v\)}, whose angle is \(\angle_vF\).
\item Every edge lies in the closure of exactly two faces, one on each side.
\end{enumerate}
\end{lemma}

\begin{proof}
For \(C\in\cA\) let \(\Pi_C\subset\R^3\) be the plane with \(C=\PP(\Pi_C)\), and let \(n_C\) be a unit normal to \(\Pi_C\). The great circles \(\Pi_C\cap S^2\) form an arrangement on \(S^2\). Each of its chambers, that is, each connected component of \(S^2\setminus\bigcup_C\Pi_C\), is \(S^2\cap\operatorname{int}K\), with closure \(S^2\cap K\), where
\[
K=\{y\in\R^3:\varepsilon_C\ip{n_C}{y}\ge0\ \text{for all }C\in\cA\}
\]
is the polyhedral cone obtained by choosing one closed half-space, with sign \(\varepsilon_C\in\{\pm1\}\), bounded by each plane \(\Pi_C\). Put \(z=\sum_C\varepsilon_Cn_C\). For \(y\in K\setminus\{0\}\), every term of \(\ip{z}{y}=\sum_C\varepsilon_C\ip{n_C}{y}\) is nonnegative, and at least one is positive because the normals span \(\R^3\), \(\cA\) being essential. Hence \(S^2\cap K\) is a compact convex spherical polygon contained in the open hemisphere \(\{\ip{z}{\cdot}>0\}\).

The quotient map \(q\colon S^2\to\RP^2\) is a local isometry and is injective on every subset of an open hemisphere. For a face \(F\), the preimage \(q^{-1}(F)\) is the union of two antipodal chambers \(\pm P\); the map \(q\) carries \(P\) homeomorphically onto \(F\) and, being closed, carries \(\overline P=S^2\cap K\) onto \(\overline F\), and the sides of \(\overline P\) lie on the great circles \(\Pi_C\cap S^2\) and map to closures of edges. Now let \(x\in\overline P\) with \(q(x)=v\) a vertex. Since \(\overline P\) lies in an open hemisphere, \(-x\notin\overline P\); hence, for a small disc \(D\) around \(v\), the set \(F\cap D\) is the image of \(P\cap D_x\), where \(D_x\) is the component of \(q^{-1}(D)\) containing \(x\). As \(\overline P\) is a convex spherical polygon, \(P\cap D_x\) is a single sector at \(x\), whose angle is the interior angle of \(\overline P\) at \(x\). Thus \(F\) contains exactly one sector at \(v\), and its angle is the interior angle at \(v\); this proves (a). Since every sector at \(v\) lies in some face, (b) follows. The same argument at an interior point of an edge, where a small disc is divided by a single line into two half-discs, proves (c).
\end{proof}

We use Girard's formula for spherical polygons; see for example \cite[Chapter~2]{Ratcliffe}. It is the special case of the Gauss--Bonnet theorem for a geodesic polygon \(P\) on the unit sphere.

\begin{lemma}[Girard]\label{lem:girard}
Let \(P\subset S^2\) be a convex spherical polygon with \(p\) vertices and interior angles \(\gamma_1,\dots,\gamma_p\). Then
\[
\operatorname{Area}(P)=\sum_{i=1}^p\gamma_i-(p-2)\pi .
\]
\end{lemma}

\section{Proof of Theorem~\ref{thm:correspondence}}\label{sec:proof-A}

\subsection{The local-to-global balance identity}\label{subsec:balance-identity}

Let \(\cA\) be an arrangement of \(n\) lines, with unit normals \(n_C\) for \(C\in\cA\). For \(C\in\cA\) let \(r_C\) be the number of vertices of \(\cA\) on \(C\). Recall from \S\ref{subsec:balanced-arr-P2} that \(M=\sum_{C\in\cA}n_Cn_C^T\) and that for a vertex \(v\), \(Z_v=\sum_{C\in\cA_v}e^{2i\theta_C}\), where \(\theta_C\) is the direction of \(C\) at \(v\). Thus \(\cA\) is balanced if and only if \(M=\frac n3I_3\), and balanced at \(v\) if and only if \(Z_v=0\).

\begin{proposition}\label{prop:balance-identity}
For every arrangement \(\cA\) of \(n\) lines in \(\RP^2\),
\begin{equation}\label{eq:balance-identity}
\sum_v|Z_v|^2=2\,\Big\|M-\frac n3\,I_3\Big\|^2+\sum_{C\in\cA}\Big(r_C-\frac{n+3}3\Big),
\end{equation}
where \(v\) runs over the vertices of \(\cA\).
\end{proposition}

\begin{proof}
Let \(v=[x]\) be a vertex and \(C,D\in\cA_v\), with directions \(\theta_C,\theta_D\) with respect to an orthonormal basis \(a,b\) of \(x^\perp\). Up to sign, \(n_C\) and \(n_D\) are the images of the tangent vectors \(\cos\theta_C\,a+\sin\theta_C\,b\) and \(\cos\theta_D\,a+\sin\theta_D\,b\) under the same rotation of \(x^\perp\) through \(\pi/2\). Therefore \(\ip{n_C}{n_D}^2=\cos^2(\theta_C-\theta_D)=\frac12\big(1+\cos2(\theta_C-\theta_D)\big)\). Summing over \(C,D\in\cA_v\) gives
\[
\sum_{C,D\in\cA_v}\ip{n_C}{n_D}^2=\frac12\big(m(v)^2+|Z_v|^2\big).
\]
Since every pair of distinct lines meets in exactly one vertex,
\[
\tr M^2=\sum_{C,D\in\cA}\ip{n_C}{n_D}^2=n+\sum_v\Big(\frac{m(v)^2+|Z_v|^2}2-m(v)\Big).
\]
Counting pairs of lines and incidences gives \(\sum_vm(v)\big(m(v)-1\big)=n(n-1)\) and \(\sum_vm(v)=\sum_Cr_C\). Since \(\tr M=n\), we have \(\tr M^2=\|M-\frac n3I_3\|^2+\frac{n^2}3\). Substituting,
\[
\sum_v|Z_v|^2=2\tr M^2-2n-n(n-1)+\sum_Cr_C=2\Big\|M-\frac n3I_3\Big\|^2-\frac{n(n+3)}3+\sum_Cr_C,
\]
which is \eqref{eq:balance-identity}.
\end{proof}

\begin{remark}
For unit vectors \(u_1,\dots,u_m\in\R^d\), Welch's inequality states that \(\sum_{j,k}\ip{u_j}{u_k}^2\ge m^2/d\), with equality precisely when the vectors form a unit-norm tight frame; see \cite[\S6.1 and Theorem~6.1]{Waldron}. The differences between the two sides of this inequality for all the normals and for those through a vertex are, respectively,
\begin{equation}\label{eq:welch-differences}
	\sum_{C,D\in\cA}\ip{n_C}{n_D}^2-\frac{n^2}{3}
	=\Big\|M-\frac n3I_3\Big\|^2, \quad
	\sum_{C,D\in\cA_v}\ip{n_C}{n_D}^2-\frac{m(v)^2}{2}
	=\frac{|Z_v|^2}{2}.
\end{equation}
Here the local inequality is applied in \(x^\perp\), where \(v=[x]\). Equation~\eqref{eq:balance-identity} relates the local and global differences in \eqref{eq:welch-differences} through the incidence numbers \(r_C\).	
\end{remark}

\begin{lemma}\label{lem:eigen}
Let \(\cA\) be a point-balanced arrangement of \(n\) lines. Then
\[
Mn_C=\tfrac12\,(n+1-r_C)\,n_C\qquad\text{for every }C\in\cA .
\]
\end{lemma}

\begin{proof}
Let \(v_1=[x_1],\dots,v_r=[x_r]\) be the vertices on \(C\), with \(r=r_C\) and \(m_i=m(v_i)\). Every line other than \(C\) meets \(C\) in exactly one of the \(v_i\), so \(\sum_i(m_i-1)=n-1\) and
\[
\sum_{i=1}^r\sum_{D\in\cA_{v_i}}n_Dn_D^T=M+(r-1)\,n_Cn_C^T .
\]
Apply both sides to \(n_C\). Since \(n_C\perp x_i\), Lemma~\ref{lem:vertex-balance}(iii) at \(v_i\) gives 
\[
Mn_C+(r-1)n_C=\sum_i\frac{m_i}2\,n_C=\frac{n-1+r}2\,n_C
\]
as we wanted to show.
\end{proof}

\begin{proposition}\label{prop:balance-criterion}
Let \(\cA\) be an irreducible arrangement of \(n\) lines. Then \(\cA\) is point-balanced if and only if \(\cA\) is balanced and
\begin{equation}\label{eq:hirzebruch-count}
r_C=\frac{n+3}3\qquad(C\in\cA).
\end{equation}
In particular, an irreducible point-balanced arrangement has \(3\mid n\) and \(n\ge6\).
\end{proposition}

\begin{proof}
If \eqref{eq:global-balance} and \eqref{eq:hirzebruch-count} hold, the right-hand side of \eqref{eq:balance-identity} vanishes, so every \(Z_v=0\) and \(\cA\) is point-balanced. This implication does not use irreducibility.

Conversely, let \(\cA\) be point-balanced. By Lemma~\ref{lem:eigen}, every \(n_C\) is an eigenvector of the symmetric matrix \(M\). If \(M\) were not scalar, its eigenspaces would be three orthogonal lines, so that \(n\le3\), or a plane \(W\) and the line \(W^\perp\), so that all lines of \(\cA\) but at most one pass through the point \([W^\perp]\). Both are excluded because \(\cA\) is essential and irreducible. So \(M\) is scalar, hence \(M=\frac n3I_3\) since \(\tr M=n\), and Lemma~\ref{lem:eigen} gives \(\frac12(n+1-r_C)=\frac n3\), which is \eqref{eq:hirzebruch-count}. Finally \(r_C\) is an integer, so \(3\mid n\), and \(n\ge4\), so \(n\ge6\).
\end{proof}

\subsection{Hirzebruch arrangements and point-balanced position}\label{subsec:proof-A}

\begin{lemma}\label{lem:stability}
Let \(\cA\) be a Hirzebruch arrangement of \(n=3k\) lines with \(k\ge2\). Then every point of \(\RP^2\) lies on at most \(k+1<\frac{2n}3\) lines of \(\cA\). Moreover, \(\cA\) is irreducible.
\end{lemma}

\begin{proof}
Let \(v\in\RP^2\) lie on \(m\ge2\) lines of \(\cA\). Since \(\cA\) is essential, there is \(C'\in\cA\) with \(v\notin C'\). The \(m\) lines through \(v\) meet \(C'\) in \(m\) distinct vertices, and \(C'\) contains exactly \(k+1\) vertices, so \(m\le k+1<2k=\frac{2n}3\). A near-pencil of \(3k\) lines has a vertex of multiplicity \(3k-1>k+1\), so \(\cA\) is irreducible.
\end{proof}

\begin{proof}[Proof of Theorem~\textup{\ref{thm:correspondence}}]
Let \(\cA\) be an irreducible arrangement of \(n\) lines. 

\emph{(i)\(\Rightarrow\)(ii).} Suppose that \(\cA\) has the Hirzebruch property, and write \(n=3k\). Since \(\cA\) is irreducible, \(k\ge2\). Lemma~\ref{lem:stability} and Proposition~\ref{prop:kempf-ness} give \(F\in\SL\) such that \(F(\cA)\) is balanced. Every line of \(F(\cA)\) contains \(k+1=\frac{n+3}3\) vertices, so \(F(\cA)\) is point-balanced by Proposition~\ref{prop:balance-criterion}.

\emph{(ii)\(\Rightarrow\)(i).} Suppose that \(F(\cA)\) is point-balanced for some \(F\in\GL\). Proposition~\ref{prop:balance-criterion} gives \(3\mid n\) and shows that every line of \(F(\cA)\) contains \(\frac{n+3}3\) vertices. The same holds for \(\cA\), which therefore has the Hirzebruch property.
\end{proof}

\subsection{Uniqueness of point-balanced position}\label{subsec:uniqueness}

\begin{corollary}\label{cor:uniqueness-hirzebruch}
Let \(\cA\) be an irreducible Hirzebruch arrangement, and let \(F_0\in\GL\) be such that \(F_0(\cA)\) is point-balanced. Then \(F(\cA)\) is point-balanced if and only if \(F\in\R^\times O(3)\,F_0\). Equivalently, there is a Euclidean inner product on \(\R^3\), unique up to scale, for whose round metric \(\cA\) is point-balanced.
\end{corollary}

\begin{proof}
One direction is clear, since \(O(3)\) acts on \(\RP^2\) by isometries. Conversely, if \(F(\cA)\) is point-balanced, then \(F_0(\cA)\) and \(F(\cA)\) are irreducible and point-balanced, hence balanced by Proposition~\ref{prop:balance-criterion}, and Proposition~\ref{prop:uniqueness} gives \(F\in\R^\times O(3)\,F_0\).

The arrangement \(\cA\) is point-balanced for the round metric of an inner product \(\beta\) on \(\R^3\) if and only if \(F(\cA)\) is point-balanced for any \(F\in\GL\) with \(\beta=\ip{F\cdot}{F\cdot}\). The inner products \(\ip{F\cdot}{F\cdot}\) with \(F\in\R^\times O(3)\,F_0\) are exactly the positive multiples of \(\ip{F_0\cdot}{F_0\cdot}\). Such an \(F_0\) exists by Theorem~\ref{thm:correspondence}.
\end{proof}

\section{Proof of Theorem~\ref{thm:balanced}}\label{sec:proof-B}

For the geometric argument in \S\S\ref{subsec:simpliciality}--\ref{subsec:equal-spacing}, fix an irreducible point-balanced arrangement \(\cA\) in \(\RP^2\). A vertex \(v\) is \emph{double}, \emph{triple} or \emph{high} if \(m(v)=2\), \(m(v)=3\) or \(m(v)\ge4\), respectively. We use Lemma~\ref{lem:cells} without further reference. In particular, the closure of a face \(F\) is the injective, locally isometric image of a convex spherical polygon, whose interior angle at a vertex \(v\) is \(\angle_vF\); the angles \(\angle_vF\) of the faces having \(v\) as a vertex are exactly the angles between consecutive lines at \(v\); and every edge lies in the closures of exactly two faces.

\subsection{Angle estimates and simpliciality}\label{subsec:simpliciality}

\begin{lemma}\label{lem:balanced-pencils}
Let \(\ell_1,\dots,\ell_m\subset\R^2\), \(m\ge2\), form a balanced pencil, and let \(g_1,\dots,g_m\) be the angles between consecutive lines.
\begin{enumerate}[label=\textup{(\alph*)}]
\item If \(m=2\), then the lines are orthogonal. If \(m=3\), then \(g_1=g_2=g_3=\pi/3\).
\item Every \(g_j\le\pi/2\), with equality only for \(m=2\).
\item If \(m\ge4\), then \(\min_jg_j+\max_jg_j\le\frac\pi2\).
\item If \(m=4\), the four lines form two orthogonal pairs. In cyclic order their angles are \(\alpha,\ \frac\pi2-\alpha,\ \alpha,\ \frac\pi2-\alpha\).
\end{enumerate}
\end{lemma}

The three-line case in (a) and the decomposition into two orthogonal pairs in (d) are the planar tight-frame statements of \cite[Exercise~2.9(c),(d)]{Waldron}. We include elementary proofs together with the angle estimates.

\begin{proof}
Under the identification \eqref{eq:directions-circle}, set \(z_j=e^{2i\theta_j}\in S^1\). By Lemma~\ref{lem:balance-characterisations}, \eqref{eq:point-balance} holds, that is, \(\sum_jz_j=0\), and the angle \(g_j\) corresponds to the arc of length \(2g_j\) from \(z_j\) to \(z_{j+1}\), with indices taken cyclically.

(a) If \(m=2\), then \(z_1=-z_2\), so the two points are antipodal and the lines are orthogonal. If \(m=3\), three unit complex numbers with sum zero are the vertices of an equilateral triangle, so the three arcs between consecutive points all have length \(2\pi/3\), and \(g_1=g_2=g_3=\pi/3\).

For (b) and (c), let \(\beta=g_j\). The open arc of length \(2\beta\) between \(z_j\) and \(z_{j+1}\) (indices cyclically) contains no \(z_i\), so all the points lie on the complementary closed arc, of length \(W=2\pi-2\beta\), whose endpoints are these two consecutive points. Rotate so that this arc is centred at \(1\), and write \(z_i=e^{i\varphi_i}\) with \(|\varphi_i|\le W/2\). Then \(\sum_i\cos\varphi_i=\operatorname{Re}\sum_iz_i=0\).

(b) If \(\beta\ge\pi/2\), then \(W\le\pi\), so every \(\cos\varphi_i\ge0\), and since their sum is zero all of them vanish. Hence \(W=\pi\) and every \(z_i\) is \(\pm i\), an endpoint of the arc. The \(z_i\) are distinct, so \(m=2\).

(c) Let \(\delta=\min_ig_i\le\pi/4\), and suppose \(\beta>\pi/2-\delta\), so \(W<\pi+2\delta\). The two endpoints have real part \(\cos(W/2)>-\sin\delta\). Every other point is at arc distance at least \(2\delta\) from both endpoints, so \(|\varphi_i|\le W/2-2\delta<\pi/2-\delta\) and \(\cos\varphi_i>\sin\delta\). There are \(m-2\ge2\) of them, so \(\sum_i\cos\varphi_i>-2\sin\delta+(m-2)\sin\delta\ge0\), a contradiction. Hence every \(g_j\le\pi/2-\delta\).

(d) A pair of unit complex numbers \(\{w,w'\}\) is determined, up to order, by \(s=w+w'\) when \(s\ne0\): they are the two points of the unit circle on the perpendicular bisector of \(0\) and \(s\). If \(z_1+z_2=0\), then also \(z_3+z_4=0\). Otherwise \(\{z_3,z_4\}\) and \(\{-z_1,-z_2\}\) have the same sum \(-(z_1+z_2)\ne0\), so they coincide. In either case the \(z_j\) form two antipodal pairs, that is, the lines form two orthogonal pairs. The direction orthogonal to \(\theta_1\) is \(\theta_1+\pi/2\). It is not \(\theta_2\), since then \(\theta_3,\theta_4\in(\theta_1+\frac\pi2,\theta_1+\pi)\) could not differ by \(\pi/2\); likewise it is not \(\theta_4\). So \(\theta_3=\theta_1+\pi/2\) and \(\theta_4=\theta_2+\pi/2\), and the angles are \(\alpha,\frac\pi2-\alpha,\alpha,\frac\pi2-\alpha\) with \(\alpha=\theta_2-\theta_1\).
\end{proof}

By Definition~\ref{def:pencil-at-point} and Lemma~\ref{lem:vertex-balance}, Lemma~\ref{lem:balanced-pencils} applies to the lines of \(\cA\) through each vertex.

\begin{proposition}\label{prop:triangles}
Every face of \(\cA\) is a triangle. If \(F\) is a face and \(v\in\overline F\) is a vertex, then \(\angle_vF\le\pi/2\), with equality only if \(v\) is double. If \(v\) is double, every face having \(v\) as a vertex has angle \(\pi/2\) at \(v\), while if \(v\) is triple, every such face has angle \(\pi/3\) at \(v\).
\end{proposition}

\begin{proof}
Let \(F\) be a face with \(p\) vertices \(v_1,\dots,v_p\), and put \(\gamma_i=\angle_{v_i}F\). This is an angle between consecutive lines of the balanced pencil at \(v_i\), so \(\gamma_i\le\pi/2\) by Lemma~\ref{lem:balanced-pencils}(b). By Lemma~\ref{lem:girard}
\[
0<\operatorname{Area}(F)=\sum_i\gamma_i-(p-2)\pi\le\frac{p\pi}2-(p-2)\pi=\frac{(4-p)\pi}{2},
\]
so \(p=3\). At a double vertex Lemma~\ref{lem:balanced-pencils}(a) gives angle \(\pi/2\), and at a triple vertex it gives angle \(\pi/3\); the equality case in Lemma~\ref{lem:balanced-pencils}(b) shows that \(\angle_vF<\pi/2\) whenever \(m(v)\ge3\).
\end{proof}

\begin{lemma}\label{lem:sector-face}
If \(\rho,\rho'\) are consecutive rays at a vertex \(x\), then the face in the sector between them is the triangle with vertices \(x\), \(y(\rho)\), \(y(\rho')\).
\end{lemma}

\begin{proof}
By Lemma~\ref{lem:cells}(a),(b), the sector lies in a face \(F\) whose closure is a spherical polygon with a vertex at \(x\), and whose two sides at \(x\) are closures of the edges leaving \(x\) along \(\rho\) and \(\rho'\), that is, the segments \([x,y(\rho)]\) and \([x,y(\rho')]\) of Definition~\ref{def:rays-neighbours}. By Proposition~\ref{prop:triangles}, \(F\) is a triangle, so its vertices are \(x\), \(y(\rho)\), \(y(\rho')\).
\end{proof}

\subsection{The minimum face angle}\label{subsec:minimum-angle}

\begin{lemma}\label{lem:tworight}
No face of \(\cA\) has two double vertices.
\end{lemma}

\begin{proof}
Suppose a face \(T\) has double vertices \(u,w\) and third vertex \(x\), and let \(L\in\cA\) contain the side \([u,w]\). By Proposition~\ref{prop:triangles}, \(T\) has right angles at \(u\) and \(w\). In a lift of \(T\) to \(S^2\), the two distinct great circles through \(\hat x,\hat u\) and \(\hat x,\hat w\) meet \(\hat L\) orthogonally, so both pass through the poles \(\pm N\) of \(\hat L\), and \(\hat x=\pm N\); we may take \(\hat x=N\). Thus every line through \(x\) meets \(L\) orthogonally. Let \(\rho_1,\dots,\rho_{2m}\) be the rays at \(x\) in cyclic order, with \(T\) between \(\rho_1\) and \(\rho_2\), and let \(F_i\) be the face between \(\rho_i\) and \(\rho_{i+1}\), with vertices \(x,y(\rho_i),y(\rho_{i+1})\) by Lemma~\ref{lem:sector-face}. If \(y(\rho_i),y(\rho_{i+1})\in L\), then, with \(b=y(\rho_{i+1})\), the sector of \(F_i\) at \(b\) is bounded by \(L\) and the line \(xb\), which are orthogonal, so \(\angle_bF_i=\pi/2\) and \(b\) is double by Proposition~\ref{prop:triangles}; hence the side of \(F_{i+1}\) at \(b\) other than \([x,b]\) lies on \(L\), and \(y(\rho_{i+2})\in L\). By induction all \(y(\rho_i)\) lie on \(L\). The lifts of the \(F_i\) are the triangles with apex \(N\) and bases on the equator \(\hat L\), filling the closed hemisphere centred at \(N\), which maps onto \(\RP^2\). So the \(F_i\) are all the faces, and every line of \(\cA\) is \(L\) or passes through \(x\notin L\): \(\cA\) is a near-pencil, contradicting irreducibility.
\end{proof}

\begin{proposition}\label{prop:min-face-angle}
Suppose that \(\cA\) has a high vertex, and let
\[
\delta=\min\{\angle_vF:F\text{ is a face and }v\in\overline F\text{ is a vertex}\}.
\]
\begin{enumerate}[label=\textup{(\alph*)}]
\item \(\pi/6<\delta\le\pi/4\), and every vertex at which the minimum angle \(\delta\) occurs is high.
\item If a face \(F\) has angle \(\delta\) at a vertex \(v\), then the other two vertices of \(F\) are one double vertex and one triple vertex.
\item If \(v\) is high and \(F\) is a face having \(v\) as a vertex, then \(\delta\le\angle_vF\le\frac\pi2-\delta\); in particular, \(\angle_vF<\pi/3\).
\end{enumerate}
\end{proposition}

\begin{proof}
If \(v\) is double or triple, then every face having \(v\) as a vertex has angle \(\pi/2\) or \(\pi/3\) at \(v\), respectively, by Proposition~\ref{prop:triangles}. If \(v\) is high, the \(m(v)\ge4\) angles between consecutive lines at \(v\) sum to \(\pi\), so the smallest is at most \(\pi/4\); these are precisely the angles \(\angle_vF\) of the faces \(F\) having \(v\) as a vertex. Hence \(\delta\le\pi/4\), and every vertex at which \(\delta\) occurs is high. If \(v\) is high, Lemma~\ref{lem:balanced-pencils}(c) gives \(\delta\le\angle_vF\le\frac\pi2-\delta\) for every face \(F\) having \(v\) as a vertex.

Choose a face \(T\) and a vertex \(v\in\overline T\) such that \(\angle_vT=\delta\), let \(u,w\) be the other two vertices of \(T\), and write
\[
\beta=\angle_uT,\qquad\gamma=\angle_wT.
\] Lemma~\ref{lem:girard} gives \(\delta+\beta+\gamma>\pi\). If \(u\) were high, then \(\beta\le\pi/2-\delta\) and \(\gamma\le\pi/2\), so \(\delta+\beta+\gamma\le\pi\); the same holds for \(w\). So \(u\) and \(w\) are double or triple. They are not both double by Lemma~\ref{lem:tworight}, and not both triple since otherwise \(\delta>\pi/3\), contradicting \(\delta\le\pi/4\). Hence one is double, the other triple, and \(\delta>\pi-\frac\pi2-\frac\pi3=\frac\pi6\). With \(\delta\le\angle_vF\le\frac\pi2-\delta\), this gives \(\angle_vF<\pi/3\) in (c).
\end{proof}

\begin{proposition}\label{prop:oneright}
Every face of \(\cA\) has exactly one double vertex.
\end{proposition}

\begin{proof}
By Lemma~\ref{lem:tworight} a face has at most one. If \(\cA\) has a high vertex, Propositions~\ref{prop:triangles} and~\ref{prop:min-face-angle}(c) imply that every angle of a face at a non-double vertex is at most \(\pi/3\). Hence a face with no double vertex would have angle sum at most \(\pi\), contradicting Lemma~\ref{lem:girard}. If \(\cA\) has no high vertex, such a face would have three angles \(\pi/3\), again contradicting Lemma~\ref{lem:girard}.
\end{proof}

\subsection{Alternation and adjacent faces}\label{subsec:adjacent-faces}

\begin{definition}\label{def:rays}
A ray \(\rho\) at a vertex \(x\) is a \emph{D-ray} if its neighbour \(y(\rho)\), in the sense of Definition~\ref{def:rays-neighbours}, is a double vertex, and an \emph{N-ray} otherwise.
\end{definition}

\begin{lemma}\label{lem:alternation}
At a non-double vertex, D-rays and N-rays alternate in cyclic order. At a double vertex all four rays are N-rays. If the multiplicity is odd, the two opposite rays of each line have different types.
\end{lemma}

\begin{proof}
For consecutive rays \(\rho,\rho'\) at \(x\), the face between them has vertices \(x,y(\rho),y(\rho')\) by Lemma~\ref{lem:sector-face}, and exactly one of them is double by Proposition~\ref{prop:oneright}. If \(x\) is non-double, exactly one of \(y(\rho)\) and \(y(\rho')\) is double, so consecutive rays have opposite types and D-rays and N-rays alternate. If \(x\) is double, neither \(y(\rho)\) nor \(y(\rho')\) is double, so all four rays at \(x\) are N-rays. Number the rays \(\rho_1,\dots,\rho_{2m}\) cyclically; the two rays of a line are \(\rho_i\) and \(\rho_{i+m}\), which have different types when \(m\) is odd.
\end{proof}

\begin{lemma}\label{lem:hypotenuse}
Let \(e\) be an edge with endpoints \(x\) and \(z\), where \(x\) is non-double and \(z\) is triple, and let \(T,T'\) be the two faces whose closures contain \(e\). Then \(\angle_xT=\angle_xT'\).
\end{lemma}

\begin{proof}
Let \(L\in\cA\) be the line containing \(e\). By Proposition~\ref{prop:oneright}, the third vertex of each face is double, so \(T\) and \(T'\) are right triangles with common side \(\overline e\) opposite their double vertices.

Let \(D\) and \(D'\) be the lines through \(z\), other than \(L\), containing sides of \(T\) and \(T'\), respectively, and let \(E\) and \(E'\) be the lines through \(x\), other than \(L\), containing sides of \(T\) and \(T'\). The sectors of \(T\) and \(T'\) at \(z\) are the two sectors at \(z\) adjacent to \(e\), so \(D\ne D'\). Reflection in \(L\) fixes \(x\) and \(z\). Since \(z\) is triple, Lemma~\ref{lem:balanced-pencils}(a) implies that the three lines \(L,D,D'\) through \(z\) are equally spaced; reflection in \(L\) therefore interchanges \(D\) and \(D'\).

The vertex \(D\cap E\) of \(T\) is double, so Proposition~\ref{prop:triangles} gives \(E\perp D\). Reflection in \(L\) preserves orthogonality, hence sends \(E\) to a line through \(x\) perpendicular to \(D'\). This line is unique, because \(x\) is not the pole of \(D'\): the line \(L\in\cA_x\) meets \(D'\) at angle \(\pi/3\). The vertex \(D'\cap E'\) of \(T'\) is also double, so \(E'\perp D'\); therefore the reflection of \(E\) is \(E'\). Thus reflection in \(L\) maps \(T\) to \(T'\), and \(\angle_xT=\angle_xT'\).
\end{proof}

\subsection{Equal spacing}\label{subsec:equal-spacing}

\begin{proposition}\label{prop:equal-spacing}
The lines of \(\cA\) through every vertex are equally spaced.
\end{proposition}

\begin{proof}
\emph{Step 0.} By Lemma~\ref{lem:balanced-pencils}(a), double and triple vertices are equally spaced, so we may assume that a high vertex exists, and let \(\delta\) be as in Proposition~\ref{prop:min-face-angle}. Choose a face \(F\) and a vertex \(v\in\overline F\) such that \(\angle_vF=\delta\). By Proposition~\ref{prop:min-face-angle}(a), the vertex \(v\) is high.

\emph{Step 1: propagation.} The sector of \(F\) at \(v\) has angle \(\delta\). Let \(S\) be any sector at \(v\) of angle \(\delta\), and let \(F_S\) be the face containing \(S\), so that \(\angle_vF_S=\delta\). By Proposition~\ref{prop:min-face-angle}(b) the two vertices of \(F_S\) other than \(v\) are a double vertex and a triple vertex \(z\), so by Lemma~\ref{lem:sector-face} one of the two rays bounding \(S\), its \emph{N-ray side}, is an N-ray with neighbour \(z\). By Lemma~\ref{lem:sector-face}, the closures of \(F_S\) and of the face containing the sector on the other side of this ray both contain \([v,z]\), so Lemma~\ref{lem:hypotenuse} gives
\begin{equation}\label{eq:delta-propagation}
\text{if a sector at \(v\) has angle \(\delta\), then so does the sector across its N-ray side.}
\end{equation}

\emph{Step 2: the multiplicity of \(v\).} Let \(m=m(v)\). Since \(m\delta\le\pi\) and \(\delta>\pi/6\), we have \(m\in\{4,5\}\). Number the rays at \(v\) as \(\rho_1,\dots,\rho_{2m}\) in cyclic order, indices mod \(2m\), and let \(S_i\) be the sector between \(\rho_i\) and \(\rho_{i+1}\). Opposite sectors \(S_i\) and \(S_{i+m}\) have the same angle, and \(S_1,\dots,S_m\) realise the \(m\) angles between consecutive lines. Choose the starting ray and the orientation so that \(S_1\) has angle \(\delta\) and its N-ray side in \eqref{eq:delta-propagation} is \(\rho_2\). By Lemma~\ref{lem:alternation} the types of the \(\rho_i\) alternate, so the even rays are N-rays, and \(S_i\) has its N-ray side \(\rho_{i+1}\) if \(i\) is odd and \(\rho_i\) if \(i\) is even.

If \(m=4\), then \eqref{eq:delta-propagation} applied to \(S_1\) shows that \(S_2\) has angle \(\delta\). These are consecutive angles of a balanced pencil of four lines, so \(2\delta=\pi/2\) by Lemma~\ref{lem:balanced-pencils}(d), and all four angles equal \(\pi/4\).

If \(m=5\), \eqref{eq:delta-propagation} carries a sector \(S_i\) of angle \(\delta\) with \(i\) odd to \(S_{i+1}\), and \(S_{i+1}\) has the same angle as the opposite sector \(S_{i+6}\), whose index is again odd. Starting from \(S_1\), this gives the chain
\[
S_1\to S_2=S_7\to S_8=S_3\to S_4=S_9\to S_{10}=S_5 ,
\]
where each arrow is an application of \eqref{eq:delta-propagation} and an equality means that opposite sectors have equal angles. Hence \(S_1,\dots,S_5\) all have angle \(\delta\), and \(\delta=\pi/5\).

In both cases \(\delta=\pi/m(v)\), and \(v\) is equally spaced.

\emph{Step 3: vertices of maximal multiplicity.} If \(x\) has multiplicity \(q\), then the \(q\) angles between consecutive lines at \(x\) are precisely the angles \(\angle_xF'\) of the faces \(F'\) having \(x\) as a vertex, so they are at least \(\delta=\pi/m(v)\); they sum to \(\pi\), so
\begin{equation}\label{eq:max-multiplicity}
q\le m(v),
\end{equation}
with equality only if all of them equal \(\pi/m(v)\). Thus every vertex of multiplicity \(m(v)\) is equally spaced. If \(m(v)=4\), \eqref{eq:max-multiplicity} shows that every high vertex has multiplicity \(4\), hence is equally spaced by the equality case of \eqref{eq:max-multiplicity}; together with Lemma~\ref{lem:balanced-pencils}(a), this proves the proposition in this case.

\emph{Step 4: \(4\)-fold vertices when \(m(v)=5\).} Now \(\delta=\pi/5\), and every vertex has multiplicity at most \(5\) by \eqref{eq:max-multiplicity}. By Step~3 and Lemma~\ref{lem:balanced-pencils}(a), it remains to show that every vertex \(x\) of multiplicity \(4\) is equally spaced. Choose an N-ray at \(x\) (Lemma~\ref{lem:alternation}), and let \(z\) be its neighbour, so that \(m(z)\in\{3,4,5\}\). Let \(T,T'\) be the two faces whose closures contain \([x,z]\), and put
\[
\alpha=\angle_xT,\qquad\beta=\angle_zT,\qquad\alpha'=\angle_xT',\qquad\beta'=\angle_zT'.
\]
The sectors of \(T\) and \(T'\) at \(x\) are the two sectors at \(x\) adjacent to the chosen ray, so \(\alpha,\alpha'\) are consecutive angles at \(x\) and \(\alpha+\alpha'=\pi/2\) by Lemma~\ref{lem:balanced-pencils}(d).

If \(m(z)=5\), then \(z\) is equally spaced by the equality case of \eqref{eq:max-multiplicity}, so \(\beta=\pi/5=\delta\), and Proposition~\ref{prop:min-face-angle}(b) shows that the vertices of \(T\) other than \(z\) are a double vertex and a triple vertex. This is impossible, since \(m(x)=4\).

If \(m(z)=4\), then likewise \(\beta+\beta'=\pi/2\). By Proposition~\ref{prop:oneright}, the third vertices of \(T\) and \(T'\) are double, so \(T\) and \(T'\) are right triangles. Girard's formula gives
\[
\alpha+\beta>\frac{\pi}{2},\qquad\alpha'+\beta'>\frac{\pi}{2}.
\]
Adding these inequalities contradicts
\[
(\alpha+\alpha')+(\beta+\beta')=\pi.
\]

Hence \(z\) is triple, and Lemma~\ref{lem:hypotenuse} gives \(\alpha=\alpha'\), so \(\alpha=\alpha'=\pi/4\). By Lemma~\ref{lem:balanced-pencils}(d), all four angles at \(x\) equal \(\pi/4\).
\end{proof}

\subsection{Reflections and classification}\label{subsec:reflections}

For a projective line \(C=\PP(\Pi)\subset\RP^2\), let \(\sigma_C\) be the reflection of \(\RP^2\) in \(C\), induced by the orthogonal reflection \(s_\Pi\) of \(\R^3\) in \(\Pi\).

\begin{lemma}\label{lem:reflection-invariance}
Let \(\cA\) be an arrangement such that the lines through every vertex are equally spaced. Then \(\sigma_C(\cA)=\cA\) for every \(C\in\cA\).
\end{lemma}

\begin{proof}
Let \(D\in\cA\), \(D\ne C\), and put \(v=C\cap D=[x]\). The reflection \(\sigma_C\) fixes \(v\) and acts on the lines through \(v\) by reflecting their tangent directions in \(x^\perp\) across that of \(C\); in terms of directions at \(v\) \textup{(}\S\textup{\ref{subsec:balanced-arr-P2})} this is \(\theta\mapsto2\theta_C-\theta\). The set of directions of an equally spaced pencil containing \(C\) is invariant under \(\theta\mapsto2\theta_C-\theta\), so \(\sigma_C(D)\in\cA\). Since also \(\sigma_C(C)=C\), \(\sigma_C\) maps \(\cA\) into itself. Since \(\sigma_C\) is an involution, \(\sigma_C(\cA)=\cA\).
\end{proof}

\begin{proposition}\label{prop:coxeter}
Let \(\cA\) be an arrangement such that \(\sigma_C(\cA)=\cA\) for every \(C\in\cA\), and let \(W\subset O(3)\) be the group generated by the reflections \(s_\Pi\), where \(\PP(\Pi)\in\cA\). Then \(W\) is a finite reflection group, and \(\cA\) is precisely the reflection arrangement of \(W\).
\end{proposition}

\begin{proof}
For each \(C=\PP(\Pi)\in\cA\), choose a unit normal \(n_C\) to \(\Pi\),
and set
\[
\Phi=\{\pm n_C:C\in\cA\}.
\]
Since \(\cA\) is essential, \(\Phi\) spans \(\R^3\). Moreover,
\(s_\Pi\) induces \(\sigma_C\) on \(\RP^2\), so the hypothesis
\(\sigma_C(\cA)=\cA\) implies \(s_\Pi(\Phi)=\Phi\). Thus \(\Phi\) is a
root system in the sense of \cite[\S1.2]{Humphreys}.

Since \(W\) preserves the finite
spanning set \(\Phi\), its action on \(\Phi\) is faithful, and hence
\(W\) is finite. By \cite[Prop.~1.14]{Humphreys}, every reflection in
\(W\) has reflecting plane \(n_C^\perp\) for some \(C\in\cA\).
Consequently, \(\cA\) is precisely the reflection arrangement of \(W\).
\end{proof}

\begin{proof}[Proof of Theorem~\textup{\ref{thm:balanced}}]
Let \(\cA\) be an irreducible point-balanced arrangement. By Proposition~\ref{prop:equal-spacing}, the lines through every vertex are equally spaced. Hence \(\sigma_C(\cA)=\cA\) for every \(C\in\cA\) by Lemma~\ref{lem:reflection-invariance}. Proposition~\ref{prop:coxeter} now shows that \(\cA\) is the reflection arrangement of the finite reflection group \(W\subset O(3)\) generated by the reflections in its lines.
\end{proof}

\begin{proof}[Proof of Corollary~\ref{cor:panov}]
Let \(\cA\) be a Hirzebruch arrangement of \(n=3k\) lines. If \(k=1\), the three lines of \(\cA\) are not concurrent, so \(\cA\) is projectively equivalent to the coordinate triangle, the reflection arrangement of \((A_1)^3\). If \(k\ge2\), then \(\cA\) is irreducible by Lemma~\ref{lem:stability}. 
By the proof of Theorem~\ref{thm:correspondence}, there exists \(F\in\SL\) such that \(\cA'=F(\cA)\) is point-balanced. By Theorem~\ref{thm:balanced}, \(\cA'\) is the reflection arrangement of the finite reflection group \(W\subset O(3)\) generated by the reflections \(s_\Pi\), \(\PP(\Pi)\in\cA'\).

We claim that \(W\) acts irreducibly on \(\R^3\). Otherwise there is a nontrivial \(W\)-invariant orthogonal decomposition \(\R^3=V_1\oplus V_2\), and one of the summands, say \(V_1\), is one-dimensional. For every \(C=\PP(\Pi)\in\cA'\), both summands are invariant under \(s_\Pi\), whose \((-1)\)-eigenspace is \(\R n_C\); hence \(n_C\) lies in \(V_1\) or in \(V_2\). As \(\cA'\) is essential, the normals span \(\R^3\), so one of them, say \(n_{C_0}\), spans \(V_1\), and all the others lie in \(V_2=n_{C_0}^\perp\). Then every line of \(\cA'\) other than \(C_0\) passes through \([n_{C_0}]\), so \(\cA'\) is a near-pencil, contrary to irreducibility.

Thus \(W\) is an irreducible finite reflection group on \(\R^3\). By the classification of finite reflection groups \cite[\S2.4 and Thm.~2.7]{Humphreys}, \(W\) is of type \(A_3\), \(B_3\) or \(H_3\). Let \(W_0\subset O(3)\) be the standard reflection group of the same type. By \cite[Prop.~2.1]{Humphreys}, there is \(g\in O(3)\) such that \(gWg^{-1}=W_0\). Since conjugation carries reflecting planes to reflecting planes, \(gF(\cA)=g(\cA')\) is the reflection arrangement of \(W_0\).
\end{proof}

\begin{remark}\label{rem:vee}
Let \(\cA\) be an irreducible point-balanced arrangement of \(n\) lines. In the normalisation of \cite[Definition~2.1]{dBPV}, the vectors \(\alpha_C=\sqrt{3/n}\,n_C\), \(C\in\cA\), form an equal-length irreducible Euclidean \(\vee\)-system in \(\R^3\). Conversely, normalising the vectors of any such system gives the unit normals of an irreducible point-balanced arrangement. Thus Theorem~\ref{thm:balanced} gives an independent proof of the rank-three classification \cite[Corollary~1.4]{dBPV}.
\end{remark}

\bibliographystyle{amsplain}
\bibliography{references}

\end{document}